\documentclass[12pt]{article}

\usepackage{amsmath}
\usepackage{amssymb}
\usepackage{amsthm}
\usepackage[pdftex]{graphicx}
\usepackage{array}

\usepackage{fancyhdr}
\usepackage{setspace}
\usepackage{slashbox}
\usepackage{url}
\usepackage{epspdfconversion}
\usepackage{appendix}

\newtheorem{theorem}{Theorem}
\newtheorem{lemma}{Lemma}

\newtheorem{conjecture}{Conjecture}
\newtheorem{corollary}{Corollary}
\theoremstyle{definition}
\newtheorem{definition}{Definition}
\theoremstyle{remark}
\newtheorem*{remark}{Remark}

\begin{document}


\begin{center}
\LARGE
Chomp on Graphs and Subsets\\
\vspace{0.2in}
\normalsize
Tirasan Khandhawit and Lynnelle Ye\\
\end{center}

\begin{abstract}
The game subset take-away begins with a simplicial complex $\bigtriangleup$. Two players take turns removing any element of $\bigtriangleup$ as well as all other elements which contain it, and the last player able to move wins. Graph Chomp is a special case of subset take-away played on a simplicial complex with only vertices and edges. The game has previously only been analyzed for complete graphs, forest graphs, and very small special cases of higher-dimensional simplicial complexes. We generalize a common method of reducing some game positions to simpler ones by symmetry and provide a complete analysis of complete $n$-partite graphs for arbitrary $n$ and all bipartite graphs. Finally, we give partial results for odd-cycle pseudotrees, which are non-bipartite graphs with a single cycle.
\end{abstract}

\section{Introduction}
Combinatorial game theory is the study of games between two or more players in which every player has complete information about the state of the game while making a move. There is no element of chance, and players move sequentially rather than simultaneously. Chess and Go are classic examples of games of this type. Combinatorial game theory has numerous applications to, for example, computer algorithms and artificial intelligence.

We deal only with two-player games in this paper. Assuming two perfect players, the full information requirement means that given any position in such a game, at least one player has a deterministic strategy that will result in a draw or a win. This is true because if one player has no such strategy, then for every possible move, the other player must have a winning response. Therefore, the other player must have a deterministic winning strategy. 

If we further require the game to always end in finite time with a winner and a loser, then the winning player can theoretically be determined for any given position. Finally, if we require the game to be impartial, so that both players have the same options when faced with the same position, then each position can be categorized as a first-player win or a second-player win. (Note that chess actually satisfies neither of these conditions, since ties and infinite games are possible and White and Black have different moves given the same position.)

In practice, even combinatorial games much simpler than Chess or Go can be extremely difficult to analyze. Two such games are subset take-away and graph chomp. Subset take-away can be played on any collection of subsets of an $n$-element set $V$ satisfying the properties of a simplicial complex, described below. The two players take turns removing a set in the collection and all of its supersets, and the player who removes the last set wins. Graph chomp is played on a graph, a set of vertices with edges between some pairs of vertices. Each turn consists of removing either an edge or a vertex and all edges connected to it, and again the last player to move wins.

Graphs and their generalization to simplicial complexes appear in many fields of computer science, such as neural networks. The bipartite graphs whose corresponding game positions we solve completely below are especially prominent in engineering applications. Understanding simple games played on graphs may also aid in understanding harder games such as the above-mentioned Chess and Go.

Both games are far from being understood for arbitrary positions. Gale and Neyman~\cite{Gale} solved subset take-away for most simplicial complexes containing all subsets of size at most $k$ (for some $k<n$) with $n$ at most $7$, as well as for special cases with arbitrarily large $n$. Christensen and Tilford~\cite{Christensen} provided further analysis of the case when $k=n-1$. Draisma and van Rijnswou~\cite{Draisma}  analyzed graph chomp played on forests as well as on complete graphs. Subset take-away and graph chomp are members of the poset game family, with connections to the famously unsolved game chocolate bar chomp; see Brouwer's website~\cite{Brouwer} for more information on progress in the theory of poset games.

Here we provide a new method for reducing some positions in subset take-away that applies more generally than methods used in \cite{Christensen} and \cite{Draisma}. We then analyze several new classes of graphs, including complete $n$-partite graphs, general bipartite graphs, and pseudotrees.

In Section~\ref{prelim}, we provide preliminaries. In Section~\ref{invo}, we prove our general method for reducing game positions and apply it to the case of a complete $n$-partite graph. In Section~\ref{bipart}, we develop a theorem to give the nim-value of all bipartite graphs. In Section~\ref{grah}, we give partial results for pseudotrees with odd cycles. In Section~\ref{fork}, we summarize and suggest directions for future work. Finally, we provide conjectures about some non-bipartite graphs in Appendix A and additional partial results in Appendix B.

\section{Preliminaries}
\label{prelim}

A \emph{partially ordered set}, or \emph{poset}, is a set $X$ together with a relationship between some pairs of elements of $X$, denoted $<$. For all $a,b,c\in X$ the relationship satisfies the properties
\begin{enumerate}
\item $a\not< a$,
\item $a\not< b$ whenever $b<a$, and
\item $a<c$ whenever $a<b$ and $b<c$. 
\end{enumerate}
A poset game is any impartial, two-person game played on a poset $(X,<)$ in which the players take turns choosing an element $k$ of $X$ and removing all elements $l$ of $X$ with $k\le l$. As is normal for combinatorial games, the player who removes the last remaining element wins. Subset take-away and graph chomp are both poset games.

Subset take-away can be played on any simplicial complex. A \emph{simplicial complex} is a collection $\bigtriangleup$ of subsets of a finite set $V$ satisfying the property that for all subsets $\sigma$ in $\bigtriangleup$, if $\varphi$ is a subset of $V$ contained in $\sigma$ then $\varphi$ is also in $\bigtriangleup$. We say that $\sigma>\varphi$ if and only if $\sigma$ is a strict superset of $\varphi$. 

So, the two players take turns choosing any $\varphi$ from $\bigtriangleup$ and removing it as well as all of its supersets. A simplicial complex may be interpreted geometrically by treating $1$-sets as vertices, $2$-sets as edges connecting the vertices they contain, $3$-sets as faces bounded by the vertices and edges they contain, and so on. Figure~\ref{3simplex} provides an example where $V=\{1,2,3\}$ and $\bigtriangleup=\{\{1\},\{2\},\{3\},\{1,2\},\{1,3\},\{2,3\},\{1,2,3\}\}$. Based on this interpretation, we refer to $1$-sets as vertices and $2$-sets as edges throughout this paper.

\begin{figure}[htbp]
\begin{center}
\includegraphics[scale=0.7]{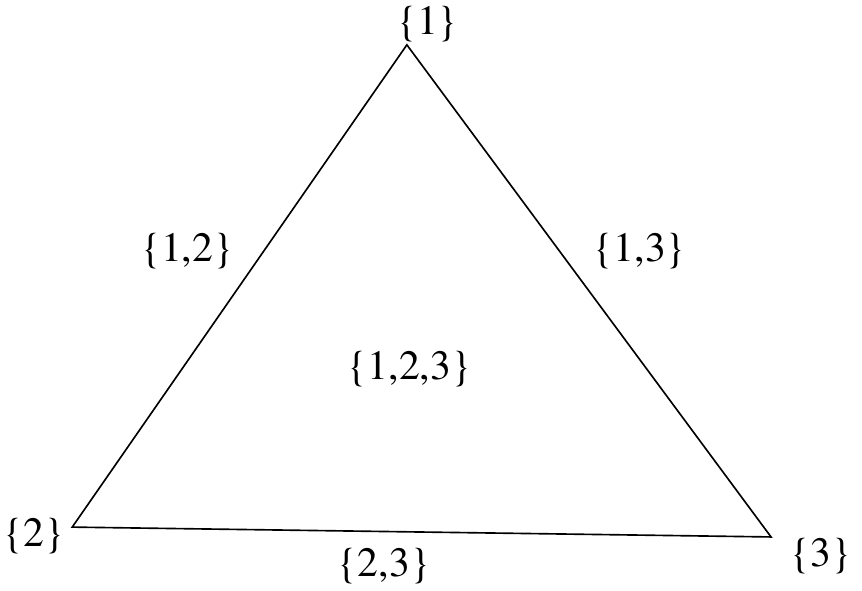}
\end{center}
\caption{Geometric interpretation of simplicial complex of all nonempty subsets of $\{1,2,3\}$.}
\label{3simplex}
\end{figure}

Graph chomp is just a special case of subset take-away where no subset in the simplicial complex has more than two elements. Both possible moves from $K_3$, the complete graph on three vertices, are shown in Figure~\ref{gchompex}.

\begin{figure}[htbp]
\begin{center}
\includegraphics[scale=0.7]{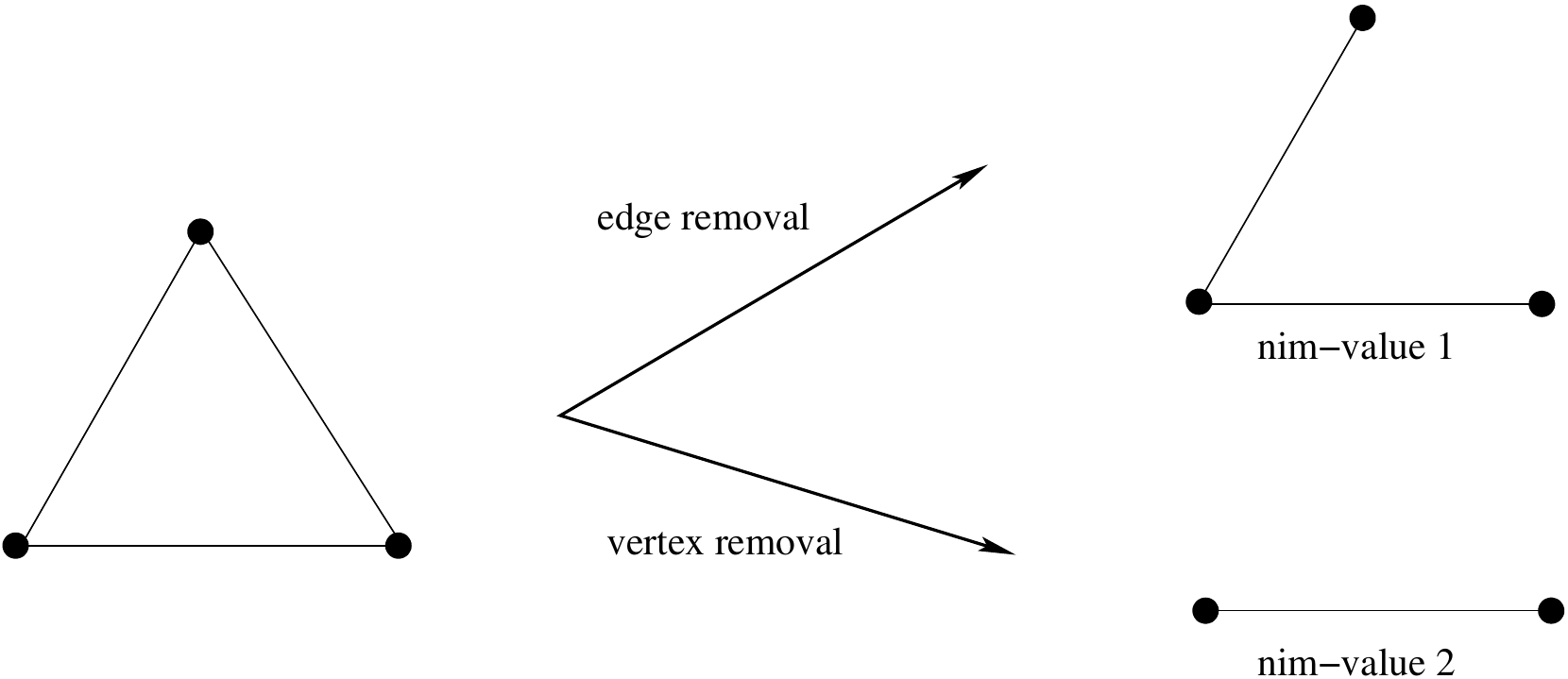}
\end{center}
\caption{Possible moves from $K_3$.}
\label{gchompex}
\end{figure}

Since only the relationships between vertices affect the game while the labels on the vertices are irrelevant, we wish to denote when two simplicial complexes are essentially the same, that is, isomorphic.

\begin{definition}
We say that simplicial complex $\bigtriangleup_1$ is \emph{isomorphic} to simplicial complex $\bigtriangleup_2$ if $\bigtriangleup_1$ can be transformed into $\bigtriangleup_2$ by a relabling of the vertices.
\end{definition}

To analyze any finite, impartial two-player combinatorial game, it is helpful to understand the model game Nim. Nim is played with several heaps of counters. The players take turns removing any positive number of counters from a single heap, and the player who takes the last counter(s) wins. Sprague~\cite{Sprague} and Grundy~\cite{Grundy} proved that any position in any finite, impartial two-player combinatorial game is equivalent to a game of Nim with one heap containing $n$ counters; the value of $n$ is called the position's nim-value. This is known as the Sprague-Grundy theorem.

A player faced with a position with nim-value 0 will lose the game, just as a player faced with an empty heap in Nim has lost. A position with nim-value 0 is also called a P-position, since the \emph{p}revious player wins. A player faced with a position with positive nim-value will win the game, just as a player faced with a single non-empty heap in Nim wins by taking all remaining counters. A position with positive nim-value is also called an N-position, since the \emph{n}ext player wins. 

The nim-value of a game is calculated as follows. The minimal nonnegative integer not included in a set of nonnegative integers is known as the \emph{m}inimal \emph{ex}cluded value, or \emph{mex}, of the set. For example, mex$\{0,1,5,8\}=2$ and mex$\{1,2,3,4\}=0$. The nim-value $g(\Gamma)$ of an arbitrary position $\Gamma$ in an arbitrary game is defined as the mex of the nim-values the next player can move to from $\Gamma$. This is known as the \emph{mex rule}. 

The mex rule implies that among the moves from a position with positive nim-value, there must be one which goes to a position with nim-value 0; on the other hand, among the moves from a position with nim-value 0, there can be none which go to another position with nim-value 0. Translating, we see that it is always possible to move from an N-position to a P-position and never possible to move from a P-position to another P-position. Therefore, determining P-positions and N-positions for a given game suffices to provide a complete strategy for the winning player, since a player faced with an N-position can always hand the opponent a P-position and the opponent must always return another N-position. Assuming the game is finite, the winning player eventually moves into an ending P-position from which the opponent cannot move.

Figure~\ref{gchompex}, for example, shows that both possible moves from $K_3$ are to path graphs, trees with only vertices of degree 1 and 2. We see later that all path graphs are N-positions, so we conclude that $K_3$ is a P-position. Numerically, the top right graph in Figure~\ref{gchompex} has a nim-value of $1$ and the bottom right graph has a nim-value of $2$, so $g(K_3)=\text{mex}\{1,2\}=0$.

Knowing the precise numerical nim-value of a position, rather than only its classification as N or P, is valuable in analyzing sums of games, in which each player can move in any one game on each turn. Given a set of games with nim-values $a_1,...,a_n$, the nim-value of their sum is $a_1\oplus\cdots\oplus a_n$ where $\oplus$ is the binary xor operation, defined as follows: the $m$th binary digit of the result is 0 if there are an even number of $a_i$ with $1$ for their $m$th binary digit and 1 if there are an odd number of $a_i$ with $1$ for their $m$th binary digit~\cite{Berlekamp}. For example, the game which combines the bottom right graph in Figure~\ref{gchompex} and a Nim-pile of two counters has nim-value 0 and is a P-position.

Draisma and van Rijnswou~\cite{Draisma} found that a forest has nim-value given by Table~\ref{forest} and that a complete graph is a P-position if and only if the number of vertices is a multiple of 3.

\begin{table}[htbp]
\begin{center}
\begin{tabular}{|c|cc|}\hline
\backslashbox{\# total vertices}{\# connected components}&even&odd\\ \hline
even&0&2\\
odd&3&1\\ \hline
\end{tabular}
\end{center}
\caption{The nim-value of a forest given the parity of the total number of vertices in all components and the parity of the number of connected components.}
\label{forest}
\end{table}

In particular, from Table~\ref{forest} we see that a single tree, such as a path graph, has nim-value $2$ if the number of vertices is even and $1$ if the number of vertices is odd. Also, a cycle always has nim-value $0$ since any move results in a path.

We introduce two more graph theory terms we use later.

\begin{definition}
A graph is \emph{bipartite} if its vertices can be split into two groups so that no two members of the same group share an edge. Equivalently, a graph is bipartite if it contains no odd cycles. 
\end{definition}

For example, all forests are bipartite, as are square grid graphs in any number of dimensions, hexagonal lattice graphs, and many far more irregular graphs.

\begin{definition}
A graph is \emph{$n$-partite} if its vertices can be split into $n$ groups so that no two members of the same group share an edge. A \emph{complete $n$-partite graph} $K_{a_1,...,a_n}$ has vertices which can be split into groups of size $a_1,...,a_n$ so that no two members of the same group share an edge and \emph{any} two members of distinct groups share an edge. The \emph{complete graph on $n$ vertices} $K_n$ is the complete $n$-partite graph with exactly one vertex in each group.
\end{definition}


\section{Reduction of Simplicial Complexes by Symmetry}
\label{invo}

In a path graph with an odd number of vertices, removing the center vertex results in two isomorphic components which can be won by copying the opponent's move in the opposite component at each turn. The winning first move is the same as if the starting position consisted only of that vertex. We found that the principle that certain kinds of symmetrical components can be added or removed without changing the outcome can be broadly extended.

\subsection{The Theorem}

\begin{definition}
For $\bigtriangleup$ a simplicial complex, we call the function $\tau:\bigtriangleup\rightarrow\bigtriangleup$ an \emph{involution} if it satisfies the following conditions:
\begin{enumerate}
\item $\tau$ restricted to the vertices of $\bigtriangleup$ is a permutation so that $\tau^2(A)=A$ for any vertex $A$.
\item $\sigma=\{\sigma_1,...,\sigma_k\}\in\bigtriangleup$ if and only if $\tau(\sigma)=\{\tau(\sigma_1),...,\tau(\sigma_k)\}\in\bigtriangleup$.
\end{enumerate}
\end{definition}

We denote the fixed point set of $\bigtriangleup$ under $\tau$ (that is, the set of elements of $\bigtriangleup$ whose images under $\tau$ are themselves) by $\bigtriangleup^{\tau}$. Essentially, an involution $\tau$ can be visualized as a way of splitting $\bigtriangleup$ into three disjoint parts $\bigtriangleup^{\tau}$, $T$, and $T'$ (none of which are necessarily simplicial complexes) so that $T$ and $T'$ are symmetric, where the symmetry is encoded in $\tau$ taking each element of $T$ to a corresponding element of $T'$ and vice versa. Draisma and van Rijnswou~\cite{Draisma} proved that if the fixed point set of a graph under an involution is also a graph, the fixed point set and the original graph are both P-positions or both N-positions. Christensen and Tilford~\cite{Christensen} extended this strategy for reducing a position to one configuration in a higher-dimensional simplicial complex. We prove a general result that encompasses both, as follows.

\begin{theorem}
Let $\bigtriangleup$ be a simplicial complex and $\tau$ be an involution on $\bigtriangleup$. Suppose the fixed point set $\bigtriangleup^{\tau}$ of $\tau$ is also a simplicial complex, so that there is no edge whose vertices are switched by $\tau$. Then $g(\bigtriangleup)=g(\bigtriangleup^{\tau})$.
\label{sym}
\end{theorem}

\begin{proof}
Given an element $\sigma$ of $\bigtriangleup$, we use the notation $\bigtriangleup-\sigma$ to mean the simplicial complex resulting from deleting $\sigma$ and all its supersets from $\bigtriangleup$.

We prove Theorem~\ref{sym} by induction. The base case, the empty graph, is trivial. Suppose the statement of the proposition is true for all sub-complexes of $\bigtriangleup$, and that $g(\bigtriangleup^{\tau})=n$. We need to prove that it is possible to move from $\bigtriangleup$ to a position with any nim-value less than $n$, and that it is impossible to move from $\bigtriangleup$ to a position with nim-value $n$.

If $\bigtriangleup^{\tau}=\bigtriangleup$, this is trivially true. Otherwise $\bigtriangleup^{\tau}\subsetneq\bigtriangleup$. Since $g(\bigtriangleup^{\tau})=n$, for any nim-value $k<n$ there exists $\sigma\in\bigtriangleup^{\tau}$ with $g(\bigtriangleup^{\tau}-\sigma)=k$. Note that $\tau$ can be restricted to an involution on $\bigtriangleup-\sigma$ so that $\bigtriangleup^{\tau}-\sigma=(\bigtriangleup-\sigma)^{\tau}$, giving $g((\bigtriangleup-\sigma)^{\tau})=k$. By the inductive hypothesis, we conclude that $g(\bigtriangleup-\sigma)=k$, so there is a move from $\bigtriangleup$ to a position with nim-value $k$ for any $k<n$.

Now suppose for the sake of contradiction that there is some $\sigma\in\bigtriangleup$ with $g(\bigtriangleup-\sigma)=n$. We have two cases: $\sigma\in\bigtriangleup^{\tau}$ and $\sigma\notin\bigtriangleup^{\tau}$. In the first case, we have by the inductive hypothesis that $g((\bigtriangleup-\sigma)^{\tau})=n$ and therefore that $g(\bigtriangleup^{\tau}-\sigma)=n$, a contradiction since $g(\bigtriangleup^{\tau})=n$. In the second case, note that because $\bigtriangleup^{\tau}$ is a simplicial complex, $\tau$ can be restricted to an involution on $\bigtriangleup-\sigma-\tau(\sigma)$ with $(\bigtriangleup-\sigma-\tau(\sigma))^{\tau}=\bigtriangleup^{\tau}$, from which $g((\bigtriangleup-\sigma-\tau(\sigma))^{\tau})=g(\bigtriangleup^{\tau})=n$. By the inductive hypothesis, we have $g(\bigtriangleup-\sigma-\tau(\sigma))=n$, again a contradiction since we assumed $g(\bigtriangleup-\sigma)=n$. We conclude that $\bigtriangleup$ has nim-value $n=g(\bigtriangleup^{\tau})$, as desired.
\end{proof}

\begin{remark}
Why Theorem~\ref{sym} requires the fixed point set to be a simplicial complex may not be immediately apparent. Abstractly, the reason is that otherwise the equation $(\bigtriangleup-\sigma-\tau(\sigma))^{\tau}=\bigtriangleup^{\tau}$ would not be valid for all $\sigma\notin\bigtriangleup^{\tau}$. 
As a concrete example of why this matters, Figure~\ref{realcountex2} shows a graph which is an N-position and has an involution which fixes edge $\{1,2\}$ and vertex $\{3\}$. Combinatorially, the fixed point set is the position $\{\{1,2\},\{3\}\}$, which, while not a simplicial complex, could be played by treating the edge as a vertex---but this gives a P-position. Thus the condition that the fixed point set is a simplicial complex is necessary.

\begin{figure}[htbp]
\begin{center}
\includegraphics{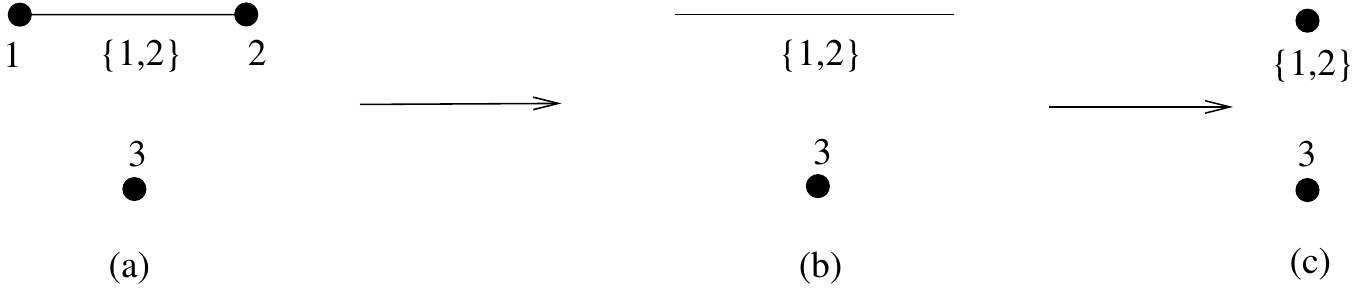}
\end{center}
\caption{(a) original graph. (b) fixed point set after involution $1\leftrightarrow2$, $3\leftrightarrow3$. (c) a representation of a nonstandard subset take-away game played on the fixed point set, created by treating the ``orphaned'' edge as a vertex.}
\label{realcountex2}
\end{figure}
\end{remark}

\begin{definition}
We say that simplicial complex $\bigtriangleup_1$ \emph{reduces to} $\bigtriangleup_2$, or $\bigtriangleup_1\downarrow\bigtriangleup_2$, if there is an involution $\tau$ on $\bigtriangleup_1$, not the identity, so that $\bigtriangleup_2$ is isomorphic to the fixed point set $\bigtriangleup_1^{\tau}$. We say that $\bigtriangleup_1$ \emph{eventually reduces to} $\bigtriangleup_2$, or $\bigtriangleup_1\Downarrow\bigtriangleup_2$, if there is a sequence $\bigtriangleup^{(1)}$, $\bigtriangleup^{(2)}$,...,$\bigtriangleup^{(k)}$ of non-isomorphic simplicial complexes so that $\bigtriangleup_1\downarrow\bigtriangleup^{(1)}$, $\bigtriangleup^{(i)}\downarrow\bigtriangleup^{(i+1)}$ for $1\le i\le k-1$, and $\bigtriangleup^{(k)}\downarrow\bigtriangleup_2$. We say that $\bigtriangleup$ is in \emph{simplest form} if it cannot be reduced to any other simplicial complex.
\label{shoehorn}
\end{definition}

\subsection{Reduction by Symmetry Example: A Triangulation of the Torus}

Suppose subset take-away is played on the simplicial complex shown in Figure~\ref{torus}. As this is a plane representation of a torus, the four corners of the square are actually the same vertex and the vertices on each edge are the same as the corresponding vertices on the opposite edge, as shown.

\begin{figure}[htbp]
\begin{center}
\includegraphics{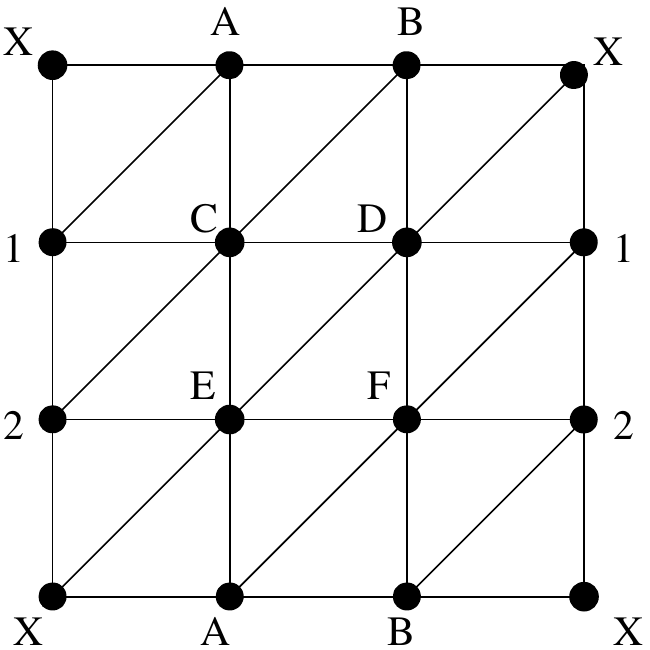}
\end{center}
\caption{A triangulation of the torus. Note that there are $18$ $2$-simplices in addition to the vertices and edges.}
\label{torus}
\end{figure}

What is the nim-value of this rather complicated picture? Theorem~\ref{sym} provides a quick answer. Consider the involution $\tau$ with $\tau(X)=X$, $\tau(D)=D$, $\tau(E)=E$, $\tau(1)=B$, $\tau(2)=A$, and $\tau(C)=F$. This corresponds to reflecting each vertex in Figure~\ref{torus} across the diagonal from the upper right to the lower left. The fixed point set is the graph with vertices $X$, $D$, $E$ and edges $\{X,D\}$, $\{D,E\}$, $\{E,X\}$, which is isomorphic to the $3$-cycle. Therefore the fixed point set is a P-position and from Theorem~\ref{sym}, the original triangulation is also a P-position. This example demonstrates the power of Theorem~\ref{sym} to simplify game positions.

\subsection{Nim-values of Complete $n$-Partite Graphs}

We now have the tools to derive the nim-value of any complete $n$-partite graph, as follows.

\begin{lemma}
The nim-value of $K_n$, the complete graph on $n$ vertices, is the residue of $n$ modulo $3$.
\label{complete}
\end{lemma}

\begin{proof}
We will prove this by induction. The statement is certainly true for $n=0,1,2$. Suppose the statement is true for all $n'<n$. There are two possible moves from $K_n$, since the next player may take a vertex or an edge. Taking a vertex results in the position $K_{n-1}$, which by the inductive hypothesis has nim-value equal to the residue of $n-1$ modulo $3$. The result of taking an edge is $K_{n-2}$ with the addition of two vertices $A$ and $B$ connected to all of the other vertices but not to each other. Then consider the involution which takes $A$ to $B$ and all other vertices to themselves. The fixed point set is the graph $K_{n-2}$, so by Theorem~\ref{sym}, $A$, $B$, and all edges incident to them may be deleted without affecting the nim-value. Therefore taking an edge results in a position with nim-value equal to the residue of $n-2$ modulo $3$ by the inductive hypothesis. The mex of these two values is the residue of $n$ modulo $3$.
\end{proof}


\begin{theorem}
The nim-value of $K_{a_1,...,a_n}$, the complete $n$-partite graph with components of sizes $a_1,...,a_n$, is equal to the residue of $p$ modulo $3$ where $p$ is the number of odd terms among $a_1,...,a_n$.
\end{theorem}

\begin{proof}
Whenever $a_i\ge2$, we take two distinct vertices $A$ and $B$ in the $i$th component and consider the involution which interchanges $A$ and $B$ and takes all other vertices to themselves. The fixed point set is the same graph with vertices $A$ and $B$ and all edges incident to them deleted; thus by Theorem~\ref{sym} we can reduce $a_i$ by two without changing the nim-value. We repeat this process until it terminates with 0 vertices in the parts where the original number of vertices was even and 1 vertex in the parts where the original number of vertices was odd. The remaining graph is $K_p$ where $p$ is the number of parts with an odd number of vertices, which by Lemma~\ref{complete} has nim-value equal to the residue of $p$ modulo $3$, as desired.
\end{proof}


\section{General Bipartite Graphs}
\label{bipart}

We now prove our theorem giving the nim-value of any bipartite graph. For the rest of this paper, we denote the number of vertices in a graph $G$ by $v(G)$ and the number of edges by $e(G)$. We begin with a two-part lemma.

\begin{lemma}
1. If $v(G)$ is odd, $G$ must include a vertex of even degree. 2. If $G$ is bipartite and $e(G)$ is odd, $G$ must include a vertex of odd degree.
\label{achievable}
\end{lemma}

\begin{proof}
To prove part 1 of Lemma~\ref{achievable}, assume $v(G)$ is odd. The sum of the degrees of all the vertices of $G$ is twice $e(G)$ and so must be even. Since $v(G)$ is odd, this sum cannot be even if all the degrees are odd, so $G$ must have a vertex of even degree. To prove part 2, assume $e(G)$ is odd. Since $G$ is bipartite, we can divide the vertices of $G$ into two groups $A$ and $B$ so that there are no edges between members of the same group. Then $e(G)$ is the sum of the degrees of the vertices in part $A$. This sum cannot be odd if all the degrees are even, so $A$ must contain a vertex of odd degree.
\end{proof}

\begin{theorem}
\label{fairest}
All bipartite graphs have nim-values as given by Table~\ref{amazing}.
\end{theorem}

\begin{proof}
We prove this by induction, since all subgraphs of a bipartite graph are also bipartite. The base case, the empty graph, is trivial. Given a graph $G$, assume the statement is true for all subgraphs of $G$. Note that deleting an edge changes the parity of $e(G)$ but not of $v(G)$, that deleting a vertex of even degree changes the parity of $v(G)$ but not $e(G)$, and that deleting a vertex of odd degree changes the parity of both $v(G)$ and $e(G)$. Based on this, we create a table of the starting parities of $v(G)$ and $e(G)$ along with their parities after each type of move, as in Table~\ref{induction}.

\begin{table}
\begin{center}
\begin{tabular}{|c|cc|}\hline
\backslashbox{v(G)}{e(G)}&even&odd\\ \hline
even&0&2\\
odd&1&3\\ \hline
\end{tabular}
\end{center}
\caption{The nim-value of a bipartite graph given the parity of the number of vertices and the parity of the number of edges.}
\label{amazing}
\end{table}

\begin{table}
\begin{center}
\begin{tabular}{|c|c|c|c|c|}\hline
starting parities&take edge&take vertex (even degree)&take vertex (odd degree)&mex\\ \hline
EE&EO (2)&OE (1)&OO (3)&0\\
OE&OO (3)&EE (0)&EO (2)&1\\
EO&EE (0)&OO (3)&OE (1)&2\\
OO&OE (1)&EO (2)&EE (0)&3\\ \hline
\end{tabular}
\end{center}
\caption{Parities of the number of vertices and the number of edges in original $G$ and after moves of each possible type; the first letter in each entry denotes the parity of the number of vertices and the second the parity of the number of edges (E for even, O for odd). The numbers in parentheses are nim-values from the inductive hypothesis.}
\label{induction}
\end{table}

Table~\ref{induction} shows that for each set of parities, $g(G)$ is at most the value shown in the ``mex'' column, consistent with Table~\ref{amazing}. To finish the induction, we need only confirm that there exists a move from each starting position in Table~\ref{induction} to a position with each lower nim-value. We check each row of Table~\ref{induction} from top to bottom.

For case EE, there is nothing to confirm; Table~\ref{induction} shows that $g(G)=0$.

For case OE, we must show that $G$ contains a vertex of even degree so that it can be moved to a position with nim-value 0. Since $v(G)$ is odd, this is true by part 1 of Lemma~\ref{achievable}.

For case EO, we must show that $G$ contains an edge so that it can be moved to a position with nim-value 0 as well as a vertex of odd degree so that it can be moved to a position with nim-value 1. Since $e(G)$ is odd, $G$ certainly does contain an edge, and it contains a vertex of odd degree by part 2 of Lemma~\ref{achievable}.

For case OO, we must show that $G$ contains an edge, a vertex of even degree, and a vertex of odd degree. Since $e(G)$ is odd, $G$ contains an edge, and $G$ contains a vertex of odd degree by part 2 of Lemma~\ref{achievable}. Since $v(G)$ is odd, $G$ contains a vertex of even degree by part 1 of Lemma~\ref{achievable}.

This completes the induction.
\end{proof}

\begin{remark}
As one would expect, Table~\ref{amazing} is consistent with the binary xor rule for adding disjoint graphs. For example, combining bipartite $G$ with $v(G)$ odd and $e(G)$ even (nim-value $1$ from Table~\ref{amazing}) and bipartite $H$ with $v(H)$ odd and $e(H)$ odd (nim-value $3$ from Table~\ref{amazing}) results in a bipartite graph with an even number of vertices and odd number of edges. Using Table~\ref{amazing} and taking $g(G)\oplus g(H)=1\oplus3$ both give a nim-value of $2$ for the combined graph.
\end{remark}

\begin{corollary}
\label{narcissus}
Suppose $G$ is a forest. If $G$ has an even number of components, $g(G)$ is $0$ if $v(G)$ is even and $3$ if $v(G)$ is odd. If $G$ has an odd number of components, $g(G)$ is $2$ if $v(G)$ is even and $1$ if $v(G)$ is odd. (This is a special case of Theorem~\ref{fairest}.)
\end{corollary}

\begin{proof}
Forests are special instances of bipartite graphs, so Theorem~\ref{fairest} is applicable. Let $G$ have $c(G)$ components. Since $G$ is a forest, we have $e(G)=v(G)-c(G)$. Thus if $c(G)$ is even, $e(G)$ and $v(G)$ have the same parity and from Table~\ref{amazing}, $g(G)$ is $0$ if $v(G)$ is even and $3$ if $v(G)$ is odd. If $c(G)$ is odd, $e(G)$ and $v(G)$ have opposite parities and again from Table~\ref{amazing}, $g(G)$ is $2$ if $v(G)$ is even and $1$ if $v(G)$ is odd.
\end{proof}

\begin{remark}
Note that Draisma and van Rijnswou~\cite{Draisma} proved Corollary~\ref{narcissus} using a less general method, as mentioned in Section~\ref{prelim}.
\end{remark}

\section{Odd-cycle Pseudotrees: A Special Case}
\label{grah} 

Non-bipartite graphs, even simple ones, exhibit much more complicated behavior than bipartite graphs. Consider a pseudotree, a graph with exactly one cycle, as in Figure~\ref{pseudotree}; it can be thought of as a cycle with trees possibly attached to each vertex in the cycle. Note that a pseudotree with an even cycle is bipartite with exactly as many edges as vertices, so it has nim-value $3$ if the number of vertices is odd and $0$ if the number of vertices is even. 

\begin{figure}[htbp]
\begin{center}
\includegraphics[scale=0.7]{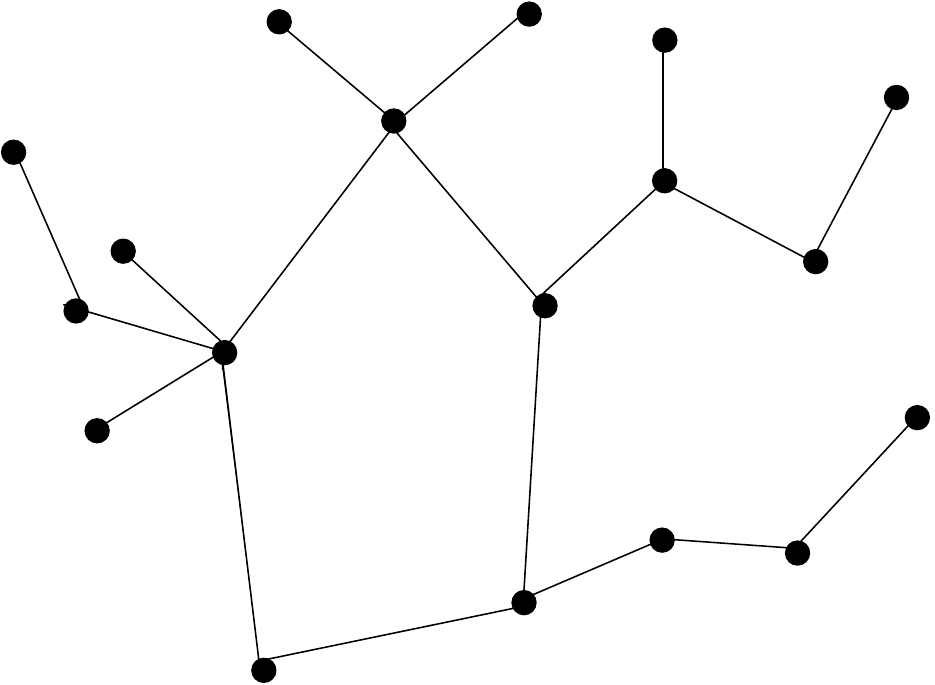}
\end{center}
\caption{A pseudotree.}
\label{pseudotree}
\end{figure}

For the rest of this section, we let $G$ be a pseudotree, $C_G$ be its cycle, and $v(G)$ be the number of vertices in $G$. We assume $G$ is in simplest form according to Definition~\ref{shoehorn} and $v(C_G)$ is odd. We now prove a way to compute some nim-values for one class of odd-cycle pseudotrees.

\begin{theorem}
Suppose vertex $A$ of $C_G$ has degree 3 and all other vertices of $C_G$ have degree $2$. Let $B$ be the vertex not in $C_G$ connected to $A$, as in Figure~\ref{ocptree}. Then $g(G)$ is at least $4$ if $B$ has odd degree, exactly $3$ if $B$ has even degree and $v(G)$ is odd, and exactly $0$ if $B$ has even degree and $v(G)$ is even.
\label{ocpt1}
\end{theorem}

\begin{figure}[htbp]
\begin{center}
\includegraphics[scale=0.8]{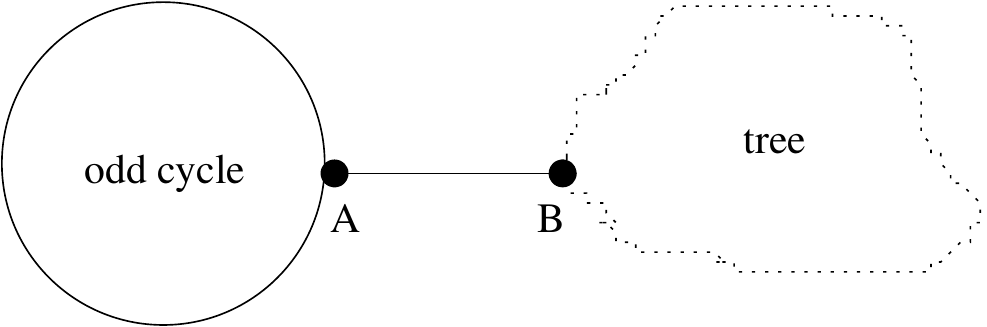}
\end{center}
\caption{An odd cycle connected to a single tree by an edge.}
\label{ocptree}
\end{figure}

\begin{proof}
Suppose $B$ has odd degree. Then, according to Corollary~\ref{narcissus}, removing $A$ results in a two-component forest with nim-value 0 if $v(G)$ is odd and 3 if $v(G)$ is even. Removing $B$ results in the union of an odd cycle and a forest with an even number of components. As the cycle's nim-value is 0 and the forest's nim-value is 3 if $v(G)$ is odd and 0 if $v(G)$ is even by Corollary~\ref{narcissus}, the resulting position has nim-value 3 if $v(G)$ is odd and 0 if $v(G)$ is even. So whatever the parity of $v(G)$, it is possible to move from $G$ to positions with nim-values 0 and 3. According to Corollary~\ref{narcissus}, it is also possible to move from $G$ to trees with nim-values 1 and 2 by removing an edge or a vertex from $C_G$ other than $A$. Thus $g(G)$ is at least 4.

Now suppose $B$ has even degree. We prove the rest of the theorem by induction. The base case when $B$ is connected to a leaf is easily verified. Assume Theorem~\ref{ocpt1} for all subgraphs of $G$ satisfying the hypotheses. We prove that it is true for $G$ by separately dealing with two cases: $v(G)$ odd and $v(G)$ even.

Case 1: $v(G)$ is odd. Deleting an edge of $C_G$ gives a tree with nim-value 1. Deleting a vertex of $C_G$ other than $A$ gives a tree with nim-value 2. Deleting vertex $A$ gives a two-component forest with nim-value 0. So $g(G)$ is not 0, 1, or 2 and we need only prove that no other moves from $G$ go to a position with nim-value 3. Deleting edge AB or vertex B results in a union of a cycle with nim-value 0 and an odd-component forest with nim-value 2 or 1, which has overall nim-value 2 or 1. Deleting an edge outside the cycle other than AB results in the disjoint sum of a pseudotree $H$ and a tree $T$. Since $g(T)$ is always 1 or 2, $g(H)$ must be 2 or 1 for a nim-sum of 3, but this is impossible since no case of the inductive hypothesis allows $g(H)$ to be 2 or 1.

Finally, deleting a vertex outside the cycle other than B results in the disjoint union of a pseudotree $H$ and a forest $F$. Since $g(H)$ is limited to 0, 3, or 4 or greater by the inductive hypothesis and $g(F)$ is limited to $\{0,1,2,3\}$, $g(H)$ and $g(F)$ should be 3 and 0 or 0 and 3 respectively to have a nim-sum of 3. In the first case $v(H)$ must be odd and $v(F)$ even, and in the second case $v(H)$ must be even and $v(F)$ odd. Thus both possibilities require an odd total number of vertices to remain after deletion of the vertex, which is impossible since $v(G)$ is odd. Thus no move takes $G$ to a position of nim-value $3$, and $G$ must have nim-value $3$.

Case 2: $v(G)$ is even. We need to show that no moves from $G$ go to positions with nim-value 0. Deleting an edge of $C_G$ gives a tree with nim-value 2. Deleting a vertex of $C_G$ other than $A$ gives a tree with nim-value 1. Deleting vertex $A$ gives a two-component forest with nim-value 3. Deleting edge AB or vertex B results in a union of a cycle with nim-value 0 and an odd-component forest with nim-value 1 or 2, which has overall nim-value 1 or 2. Deleting an edge outside the cycle other than AB results in the disjoint sum of a pseudotree $H$ and a tree $T$. Since $g(T)$ is always 1 or 2, $g(H)$ must be 1 or 2 for a nim-sum of 0, but this is impossible since no case of the inductive hypothesis allows $g(H)$ to be 1 or 2.

Finally, deleting a vertex outside the cycle other than B results in the disjoint union of a pseudotree $H$ and a forest $F$. Since $g(H)$ is limited to 0, 3, or 4 or greater by the inductive hypothesis and $g(F)$ is limited to $\{0,1,2,3\}$, we must have $g(H)=g(F)=0$ or $g(H)=g(F)=3$ to obtain a nim-sum of 0. In the first case $v(H)$ and $v(F)$ must both be even, and in the second case $v(H)$ and $v(F)$ must both be odd. Thus both possibilities require an even total number of vertices to remain after deletion of the vertex, which is impossible since $v(G)$ is even. Thus no move takes $G$ to a position of nim-value $0$, and $G$ must have nim-value $0$.
\end{proof}

\subsection{The Graph $G_{m,k}$}

In general, when $v(C_G)$ is odd, the function $g(G)$ can become extremely complicated. We now calculate exact nim-values for a very simple odd-cycle pseudotree.

\begin{definition}
In a graph, a $k$\emph{-tail} is $k-1$ vertices of degree $2$ connected in sequence, together with a leaf connected to the $(k-1)$th vertex.
\label{tail}
\end{definition}

Define $G_{m,k}$ to be a pseudotree with cycle $C_{m,k}$ such that $v(C_{m,k})$ is odd and all the vertices in $C_{m,k}$ except for a single vertex $A$ have degree 2. $A$ has degree 3 and is connected to vertex $B$ not in the cycle, and $B$ is connected to an $m$-tail and a $k$-tail as in Definition~\ref{tail}. See Figure~\ref{Gmk}.

\begin{figure}[htbp]
\begin{center}
\includegraphics[scale=0.8]{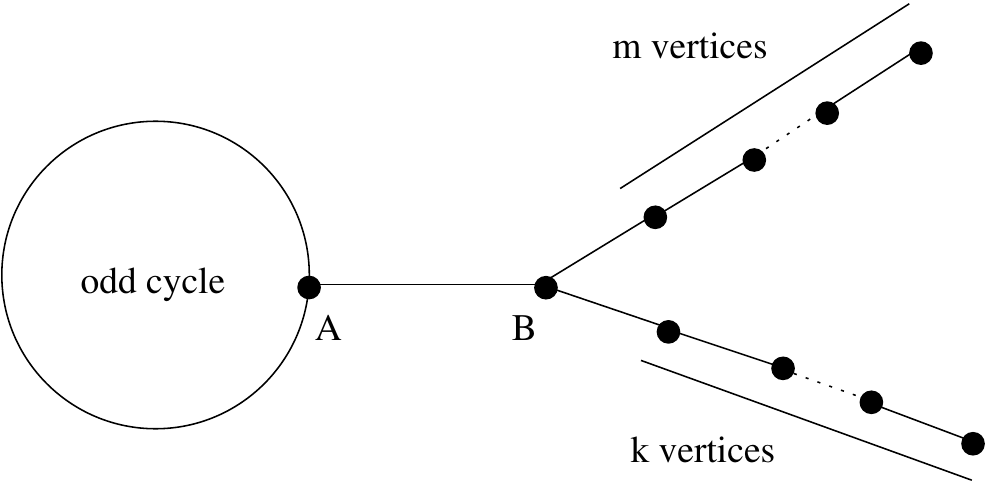}
\end{center}
\caption{The graph $G_{m,k}$, an odd cycle connected to a vertex which is connected to an $m$-tail and a $k$-tail.}
\label{Gmk}
\end{figure}

Note that in this case, the exact value of $v(C_{m,k})$ does not affect the nim-value as long as it is odd, as can be proven by a quick induction on $k$ and $m$. It is easy to check that $g(G_{0,0})=4$ regardless of the size of the cycle. Let $G'_{m,k}$ be a pseudotree with the same tail configuration as $G_{m,k}$ but a differently-sized cycle, and suppose $g(G_{m',k'})=g(G'_{m',k'})$ whenever $\{m'<m$ and $k'=k\}$ or $\{k'<k$ and $m'=m\}$. Any move in the $m$-tail of $G_{m,k}$ results in the union of $G_{m',k}$ and a path, and the corresponding move in $G'_{m,k}$ results in the union of $G'_{m',k}$ and an isomorphic path, which has the same nim-value by the inductive hypothesis. The same is true of moves in the $k$-tails, and we already showed in the proof of Theorem~\ref{ocpt1} that moves outside the tails of $G_{m,k}$ and $G'_{m,k}$ yield the same set of nim-values. Applying the mex rule, we conclude that $g(G_{m,k})=g(G'_{m,k})$.

Let $g_{m,k}=g(G_{m,k})$. Table~\ref{3by3} shows $g_{m,k}$ for small values of $m$ and $k$.

\begin{table}[htbp]
\begin{center}
\begin{tabular}{|c|ccc|ccc|ccc|ccc|}\hline
\backslashbox{$m$}{$k$}&1&2&3&4&5&6&7&8&9&10&11&12\\ \hline
1&4&6&4&8&10&8&12&14&12&16&18&16\\
2&6&4&6&10&8&10&14&12&14&18&16&18\\
3&4&6&4&8&10&8&12&14&12&16&18&16\\ \hline
4&8&10&8&4&6&4&16&18&16&12&14&12\\
5&10&8&10&6&4&6&18&16&18&14&12&14\\
6&8&10&8&4&6&4&16&18&16&12&14&12\\ \hline
7&12&14&12&16&18&16&4&6&4&8&10&8\\
8&14&12&14&18&16&18&6&4&6&10&8&10\\
9&12&14&12&16&18&16&4&6&4&8&10&8\\ \hline
10&16&18&16&12&14&12&8&10&8&4&6&4\\
11&18&16&18&14&12&14&10&8&10&6&4&6\\
12&16&18&16&12&14&12&8&10&8&4&6&4\\ \hline
\end{tabular}
\end{center}
\caption{Values of $g_{m,k}$ for small $m$ and $k$.}
\label{3by3}
\end{table}

We observe that Table~\ref{3by3} is made up of $3\times3$ boxes of a specific form, as follows.

\begin{theorem}
\label{boxy}
For any nonnegative integers $a$ and $b$, we have
$$\begin{pmatrix}g_{3a+1,3b+1}&g_{3a+1,3b+2}&g_{3a+1,3b+3}\\
g_{3a+2,3b+1}&g_{3a+2,3b+2}&g_{3a+2,3b+3}\\
g_{3a+3,3b+1}&g_{3a+3,3b+2}&g_{3a+3,3b+3}\end{pmatrix}
=\begin{pmatrix}4n&4n+2&4n\\4n+2&4n&4n+2\\4n&4n+2&4n\end{pmatrix}$$
where $n=(a\oplus b)+1$. We refer to this as a $4n$-block.
\end{theorem}

\begin{proof}
Let $S_{m,k}$ be the set $\{g_{i,k}\oplus1|0\le i\le m-2\}\cup\{g_{i,k}\oplus2|0\le i\le m-2\}\cup\{g_{m,j}\oplus1|0\le j\le k-2\}\cup\{g_{m,j}\oplus2|0\le j\le k-2\}$. From the definition of the nim-value, we have
\begin{center}
$\displaystyle g_{m,k}=\text{mex}\{S_{m,k}\cup\{0,1,2,3,\hspace{0.1cm}g_{m-1,k},\hspace{0.1cm}g_{m-1,k}\oplus1,\hspace{0.1cm}g_{m,k-1},\hspace{0.1cm}g_{m,k-1}\oplus1\}\}$
\end{center}
where the inclusion of 0, 1, 2, and 3 is justified by Theorem~\ref{ocpt1}.

We first prove by induction that Theorem~\ref{boxy} holds for some $n$ and then prove that $n=(a\oplus b)+1$. For the induction, the base case $a=b=0$ is easily computed. Assume the box pattern holds for $a',b'$ whenever $\{a'<a$ and $b'=b\}$ or $\{b'<b$ and $a'=a\}$. First we prove that $g_{3a+1,3b+1}$ is a multiple of 4. Note that by the inductive hypothesis, all $g_{m',3b+1}$ and $g_{3a+1,k'}$ for $m'<3a+1$, $k'<3b+1$ are of the forms $4x$ and $4x+2$ for some integer $x$, and furthermore $4x$ is represented in $\{g_{m',3b+1}|m'<3a\}\cup\{g_{3a+1,k'}|k'<3b\}$ if and only if $4x+2$ is represented. So if $4x$ is represented, then $4x\oplus1=4x+1$, $4x\oplus2=4x+2$, $(4x+2)\oplus1=4x+3$, and $(4x+2)\oplus2=4x$ all appear in $S_{3a+1,3b+1}$. This also means that $g_{3a,3b+1}$, $g_{3a,3b+1}\oplus1$, $g_{3a+1,3b}$, and $g_{3a+1,3b}\oplus1$ are already in $S_{3a+1,3b+1}$ and do not place further constraints on $g_{3a+1,3b+1}$. Therefore, if $4y+i$ does not appear in $S_{3a+1,3b+1}$ for some $i\in\{1,2,3\}$, neither does $4y$. We conclude that the mex of $S_{3a+1,3b+1}\cup\{0,1,2,3\}$, and therefore $g_{3a+1,3b+1}$, is a multiple of 4. Specifically, it is the least positive multiple of 4 which does not appear earlier in the row or column.

Now suppose $g_{3a+1,3b+1}=4y$. By the inductive hypothesis, the set $\{g_{m',3b+1}|m'<3a+1\}$ is identical to the set $\{g_{m',3b+2}|m'<3a+1\}$. That is, $g_{3a+1,3b+2}$ has to satisfy the same constraints as $g_{3a+1,3b+1}$, and so must be at least $4y$, with the additional constraint that it cannot be $g_{3a+1,3b+1}=4y$ or $g_{3a+1,3b+1}\oplus1=4y+1$. We conclude that $g_{3a+1,3b+2}=4y+2$. Similarly, $g_{3a+2,3b+1}=4y+2$. Now $g_{3a+3,3b+1}$ must also satisfy the same constraints as $g_{3a+1,3b+1}$ with the additional conditions that it cannot be $g_{3a+1,3b+1}\oplus1=4y+1$, $g_{3a+1,3b+1}\oplus2=4y+2$, $g_{3a+2,3b+1}=4y+2$, or $g_{3a+2,3b+1}\oplus1=4y+3$; the minimum such value is $4y$. Therefore, $g_{3a+3,3b+1}=4y$ and similarly, $g_{3a+1,3b+3}=4y$. Also, by considering the constraints placed by $g_{3a+2,3b+1}$ and $g_{3a+1,3b+2}$ in a similar manner, we can show that $g_{3a+2,3b+2}=4y$. Next we can fill in $g_{3a+3,3b+2}=g_{3a+2,3b+3}=4y+2$ and finally $g_{3a+3,3b+3}=4y$. This proves the inductive step.

Note that the proof of the inductive hypothesis is indeed valid when $a=0$ or $b=0$, even though some of the sets referred to become empty sets, because our comparisons between constraint sets do not change.

\begin{table}[htbp]
\begin{center}
\begin{tabular}{|c|cc|cc||cc|cc|}\hline
\backslashbox{a}{b}&0&1&2&3&4&5&6&7\\ \hline
0&4&8&12&16&20&24&28&32\\
1&8&4&16&12&24&20&32&28\\ \hline
2&12&16&4&8&28&32&20&24\\
3&16&12&8&4&32&28&24&20\\ \hline\hline
4&20&24&28&32&4&8&12&16\\
5&24&20&32&28&8&4&16&12\\ \hline
6&28&32&20&24&12&16&4&8\\
7&32&28&24&20&16&12&8&4\\ \hline
\end{tabular}
\end{center}
\caption{Values of $g_{3a+1,3b+1}$ for small $a$ and $b$.}
\label{nimsum}
\end{table}

Since knowing the value of $g_{3a+1,3b+1}$ determines $g_{3a+i,3b+j}$ for $1\le i,j\le3$, we need only consider the pattern of the $4n$-blocks as $a$ and $b$ vary, as in Table~\ref{nimsum}. The value $g_{3a+1,3b+1}$ is just the least positive multiple of 4 which does not appear earlier in row $a$ or column $b$. Now consider a similar table in which the row number represents the size of a Nim-pile, the column number represents the size of a second Nim-pile, and the entry at a certain position represents the nim-value of the two piles combined. From a given two-pile position the possible moves are exactly those to positions with one pile unchanged and the other decreased, so by the mex rule, the entry at a position would simply be the mex of the earlier entries in its row and column. Thus, after dividing every entry in Table~\ref{nimsum} by 4 and subtracting 1, we see that the construction of Table~\ref{nimsum} is equivalent to the construction of a nim-addition table for two piles of size $a$ and $b$. But we know the nim-addition of $a$ and $b$ results in $a\oplus b$. Therefore, $g_{3a+1,3b+1}=4[(a\oplus b)+1]$, as desired.
\end{proof}

\begin{remark}
Upon division by $2$, the top row in Table~\ref{3by3} bears an intriguing similarity to the nim-sequence of the octal game $.3\overline{4}$, which begins with a pile of counters. On each turn, a player may either remove one counter from one pile or remove multiple counters from one pile and split the rest of the pile into two nonempty piles. The game has nim-sequence $101232454676...$ where the $i$th number refers to the nim-value of a pile of $i$ counters~\cite{Berlekamp}. While graph chomp on an odd-cycle pseudotree with long tails does have octal-game-like characteristics, we do not know of a way to make a concrete connection.
\end{remark}

Though Theorem~\ref{boxy} shows that $G_{m,k}$ is still somewhat well-behaved, more complicated odd-cycle pseudotrees seem to have more erratic nim-values. See Appendix A for interesting examples. We do not currently know a way to analyze them more generally.

\section{Summary and Future Work}
\label{fork}

We found a method of reducing simplicial complexes that applies to many game positions that had not previously been analyzed. We generalized previous results for complete graphs to all complete $n$-partite graphs and computed their nim-values. We also generalized previous results for forest graphs to all bipartite graphs and computed their nim-values. Finally, we analyzed some simple non-bipartite graphs.

There is much more analysis to be done for non-bipartite graphs. We have seen that even extremely simple non-bipartite graphs with a single cycle have very rich structures of nim-values. In Appendix A, we present conjectures for the nim-values of more general odd-cycle pseudotrees as well as the wheel graph $W_n$. 
Also, subset take-away in higher dimensions has been relatively unstudied. Analysis becomes especially difficult when higher-dimensional sets appear in a game position. Move options increase, and some moves yield positions that we are not yet able to analyze. For example, removing the interior of a $3$-set (which must appear in a simplicial complex which is not a graph) immediately results in an odd cycle, which is already too complicated with only vertices and edges in the picture.

Finally, we believe the connection between subset take-away and better-studied poset games such as regular Chomp has not yet been fully explored. It may be possible to use results for one to help understand results for the other. Considering the difficulty of regular Chomp and similar poset games, any such developments would be extremely valuable.

\section{Acknowledgments} 

We would like to thank Sira Sriswasdi of the University of Pennsylvania for creating computer programs to generate useful data, Dr. John Rickert of the Rose-Hulman Institute of Technology for his help with research writing and presentation, Dr. Tanya Khovanova of MIT for coordinating the RSI 2009 math projects and for helpful discussion, and Dr. Dan Christensen of the University of Western Ontario for helpful discussion. Finally, we would like to thank the S.D. Bechtel, Jr. Foundation, the Center for Excellence in Education, and MIT for sponsoring us in the summer of 2009.


\appendix

\section{Conjectures}

Let $G$ be a pseudotree, $C_G$ be its cycle, and $v(G)$ be the number of vertices in $G$. We assume $G$ \textit{is in simplest form} and $v(C_G)$ is odd.

\begin{conjecture}
Suppose vertex $A$ of $C_G$ has degree at least $3$ and all other vertices of $C_G$ have degree $2$. Take some leaf $X$ of $G$; let $G_0=G$ and $G_k$ be the graph constructed by adjoining a $k$-tail to vertex $X$ for $k\ge1$. Consider the sequence of nim-values $\{g_k\}_{k=0,1,...}$ where $g_k=g(G_k)$. We conjecture that $\{g_k\}$ eventually falls into one of two patterns. Either $\{g_k\}$ is eventually periodic with period $2$, alternating between values of the form $4n$ and $4n+3$ ($n$ integral), or it eventually matches some row of Table~\ref{3by3} (that is, it eventually looks something like $4,6,4,8,10,8,...$). In particular, no such pseudotree has a nim-value congruent to $1\pmod 4$.
\label{ocpt2}
\end{conjecture}

\begin{remark}
Graphs exhibiting the $4,6,4,8,10,8,...$ pattern as well as the $4n,4n+3,4n,4n+3,...$ pattern for $n=1,2,3$ are pictured in Table~\ref{7-11-15}. The $n=0$ case is covered by Theorem~\ref{ocpt1}.
\end{remark}

\begin{table}[htbp]
\begin{center}
\begin{tabular}{cccc}
\includegraphics[scale=0.8]{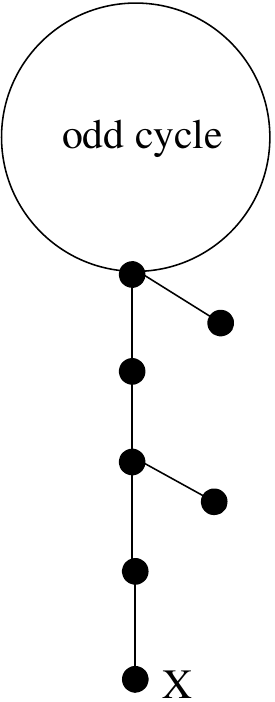} & \includegraphics[scale=0.8]{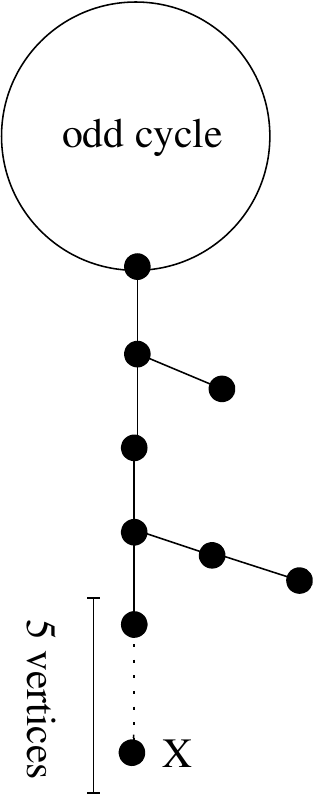} & \includegraphics[scale=0.8]{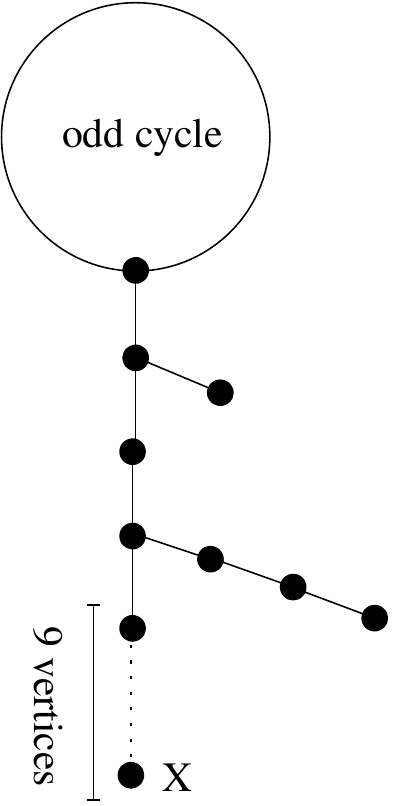} & \includegraphics[scale=0.8]{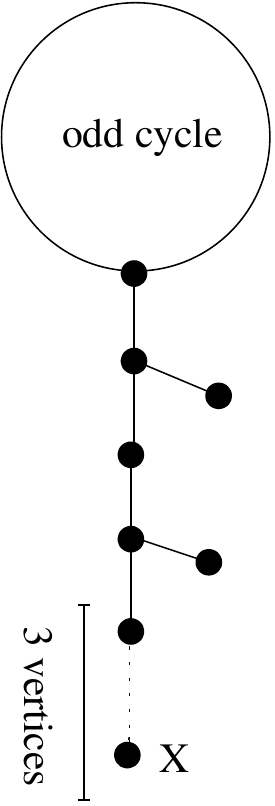}\\
nim-value 7 & nim-value 11 & nim-value 15 & nim-value 8\\
A & B & C & D
\end{tabular}
\caption{Adjoining tails of increasing length to the vertex labeled $X$ in graphs A, B, C, and D result respectively in the nim-sequences $4,7,4,7,...$; $8,11,8,11...$; $12,15,12,15...$; and $10,8,12,14,12,16,18,16,...$.}
\label{7-11-15}
\end{center}
\end{table}

\begin{conjecture}
Suppose there are two or more vertices of $C_G$ with degree at least $3$. Then $g(G)$ is $3$ if $|v(G)|$ is odd and $0$ if $|v(G)|$ is even.
\label{ocpt4}
\end{conjecture}

\begin{remark}
While it seems plausible that attaching trees to multiple vertices of $C_G$ has a ``shielding'' effect which prevents the accumulation of ugly nim-values, and small examples do turn out this way, this conjecture is also somewhat optimistic.
\end{remark}

\begin{conjecture}
\label{wheel}
The wheel graph $W_n$, consisting of an $n$-cycle and a center vertex connected to each vertex in the cycle as in Figure~\ref{W6}, has nim-value $1$ for all $n\ge3$.
\end{conjecture}

\begin{figure}[htbp]
\begin{center}
\includegraphics{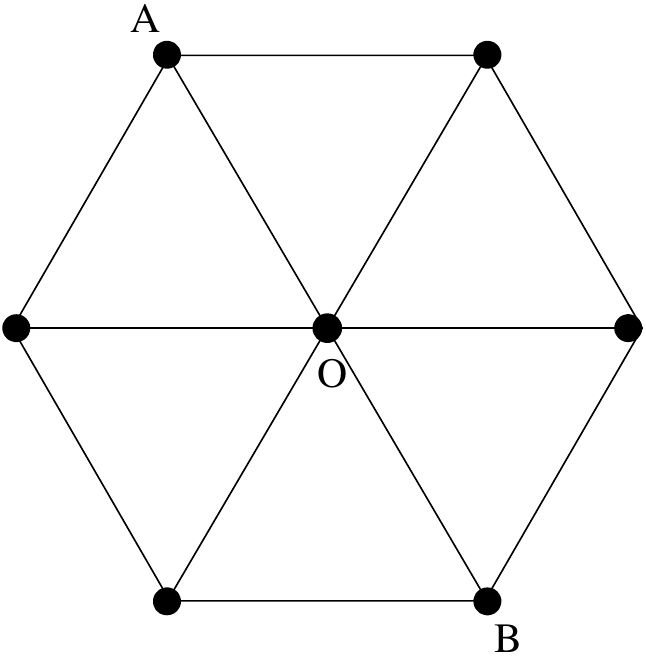}
\end{center}
\caption{The wheel graph $W_6$.}
\label{W6}
\end{figure}

\begin{remark}
Note that $W_n$ is an N-position for all $n\ge3$ since taking the center vertex results in a cycle. Also, Conjecture~\ref{wheel} is certainly true for even $n$, as then we can label the center vertex $O$ and any diametrically opposite pair of vertices $A$ and $B$ as in Figure~\ref{W6} and consider the involution which takes $A$, $B$, $O$ to themselves and all other vertices to the diametrically opposite one. Then the fixed point set is the path with $3$ vertices, which has nim-value $1$. It is also not difficult to verify the conjecture for $n=3,5,7$ by hand. For odd $n$ greater than $7$ calculations are no longer manageable.
\end{remark}

\section{Odd-cycle Hairballs}

\begin{definition}
We refer to a pseudotree $G$ as a \emph{hairball} if no vertex not in $C_G$ has degree greater than $2$. See Figure~\ref{hairball} for an example.
\end{definition}

Like $G_{m,k}$, the hairballs are a relatively well-behaved class of pseudotree, and the following theorem completely classifies their nim-values.

\begin{theorem} 
\label{HAIR}
Suppose $G$ is a hairball with $v(C_G)$ odd which is in simplest form and is not just a cycle. Then $g(G)$ is determined by the following rules:
\begin{enumerate}
\item If $v(G)$ is even, all vertices in $C_G$ other than vertex $A$ have degree $2$, and vertex $A$ has degree $4$ and is connected to a $k$-tail and a $(k+1)$-tail for some nonnegative $k$, then $g(G)=4$.
\item If $v(G)$ is even and $G$ is not in the above form, then $g(G)=0$.
\item If $v(G)$ is odd, then $g(G)=3$.
\end{enumerate}
\end{theorem}

\begin{figure}[htbp]
\begin{center}
\includegraphics{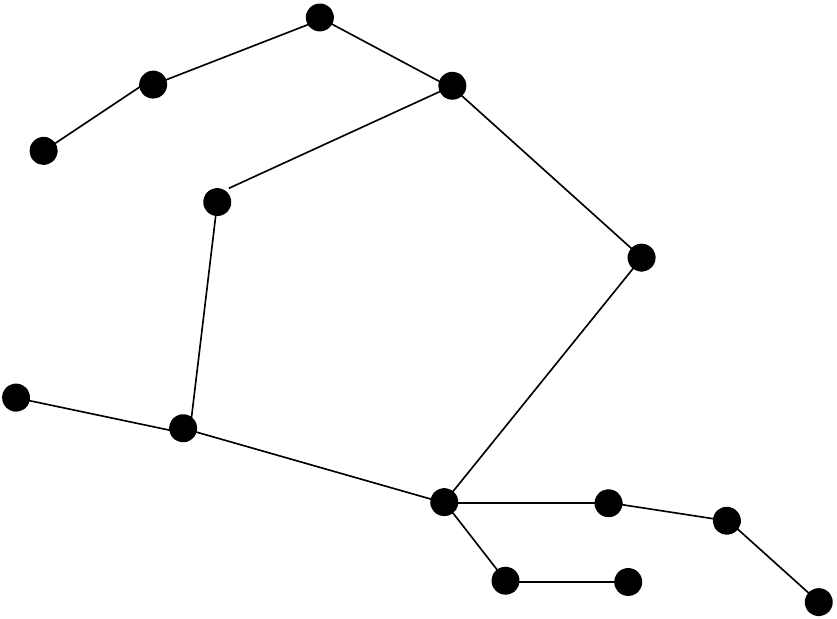}
\end{center}
\caption{An example of a hairball graph.}
\label{hairball}
\end{figure}

\begin{proof}
We prove this by strong induction. It is easy to verify three base cases, shown in Figure~\ref{basecase}:
\begin{enumerate}
\item $G$ consists of $C_G$ and a leaf attached to a vertex of $C_G$,
\item $G$ consists of $C_G$ and a $2$-tail attached to a vertex of $C_G$, and
\item $G$ consists of $C_G$ a leaf attached to each of two vertices of $C_G$.
\end{enumerate}

\begin{figure}[htbp]
\begin{center}
\includegraphics[width=\textwidth]{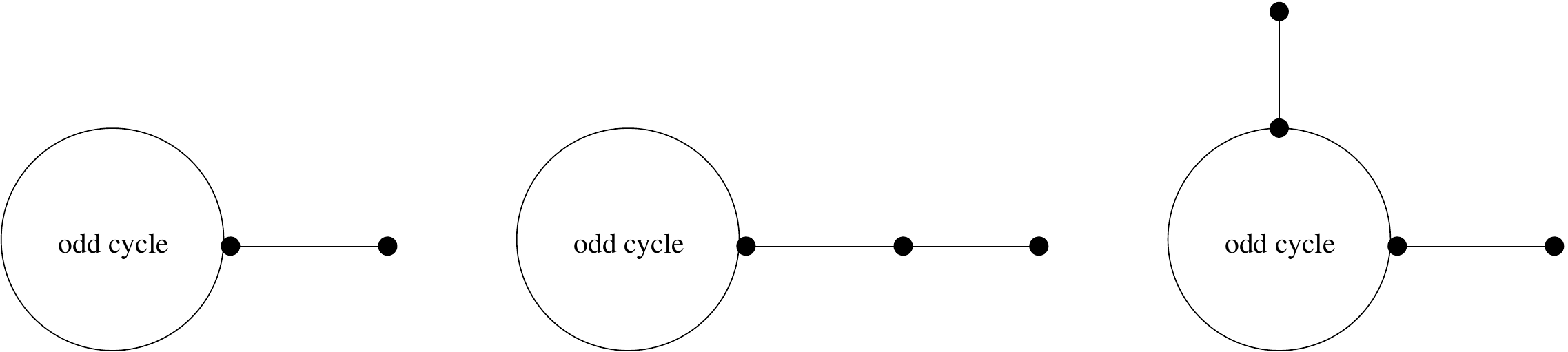}
\end{center}
\caption{Base case configurations.}
\label{basecase}
\end{figure}

Given a graph $G$ satisfying the hypotheses of Theorem~\ref{HAIR} which is not one of the base cases, assume the statement is true for all of its subgraphs.

First suppose $G$ is of the form described in part 1 of Theorem~\ref{HAIR}. $G$ can be moved to a position with nim-value $0$ by deleting the leaf of the $(k+1)$-tail, $1$ by deleting an edge of $C_G$, $2$ by deleting a vertex of $C_G$, and $3$ by deleting the leaf of the $k$-tail and the inductive hypothesis; another position of the same form, which would have nim-value $4$, is not achievable. Therefore $G$ has nim-value $4$.

Next suppose $G$ is of the form described in part 2 of Theorem~\ref{HAIR}. We must show that $G$ cannot be moved to a position with nim-value $0$. Deleting an edge of $C_G$ gives $2$ and deleting a vertex of $C_G$ gives a bipartite graph with an odd number of vertices. Deleting a leaf does not result in a position with nim-value $0$ since $G$ is not of the form described in part 1. Deleting any other edge or vertex results in the combination of a smaller hairball (possible nim-values $0,3,4$ by the inductive hypothesis) and a path graph ($1$ or $2$), which is also not $0$.

Finally, suppose $G$ is of the form described in part 3 of Theorem~\ref{HAIR}. We first need to find moves from $G$ to positions with nim-values $0$, $1$, and $2$. $G$ can be moved to $0$ by deleting any leaf, unless $G$ falls into one of the following categories, shown in Figure~\ref{special} with $k=2$, $m=0$:
\begin{enumerate}
\item $G$ consists of $C_G$ and a $k$-tail and a $(k+2)$-tail attached to one vertex of $C_G$.
\item $G$ consists of $C_G$ and an $m$-tail, an $(m+1)$-tail, a $k$-tail, and a $(k+1)$-tail, all of distinct lengths, attached to one vertex of $C_G$.
\item $G$ consists of $C_G$, an $m$-tail and an $(m+1)$-tail attached to one vertex of $C_G$, and a $k$-tail and a $(k+1)$-tail attached to another vertex of $C_G$.
\end{enumerate}
Each of these positions can be moved to a position with nim-value $0$ by deleting the leaf of the second-longest tail, since the result cannot be of the form described in part 1 of Theorem~\ref{HAIR}. 

\begin{figure}[htbp]
\begin{center}
\includegraphics{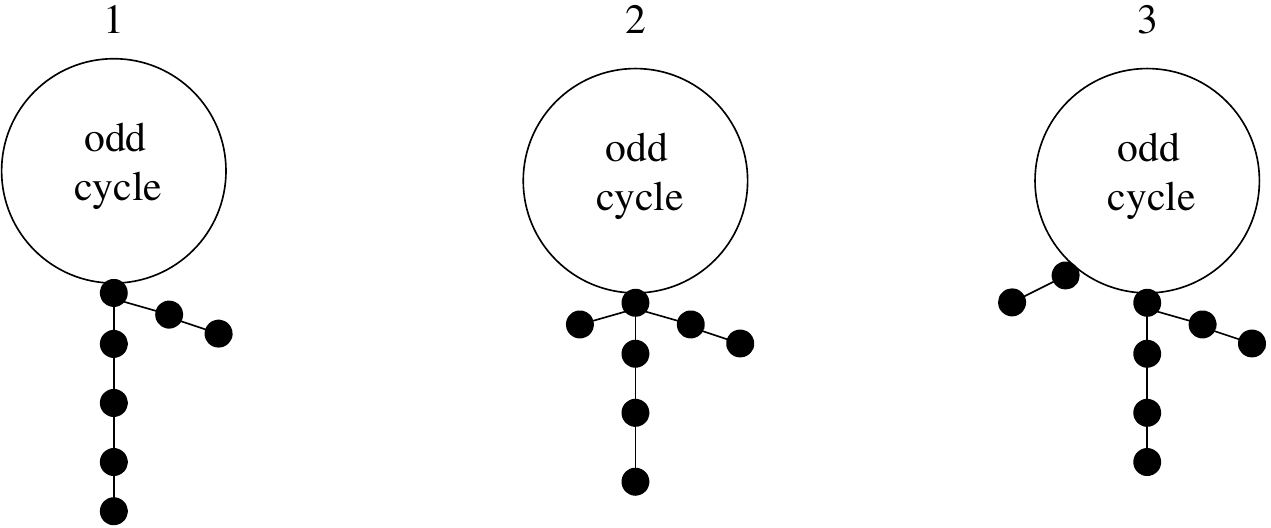}
\end{center}
\caption{The three special cases with $k=2$ and $m=0$.}
\label{special}
\end{figure}

$G$ can be moved to $1$ by deleting an edge in $C_G$. $G$ can be moved to a position with nim-value $2$ by deleting any vertex in $C_G$ with even degree if one exists. If there are no such vertices, $G$ must contain a tail of length at least 2; otherwise all vertices outside the cycle would be leaves, so there would be an odd number of vertices outside the cycle and an even total number of vertices. Then take the second-to-last vertex in that tail. The result is the disjoint union of a vertex and a pseudotree which cannot reduce to the odd cycle since all but one vertex in $C_G$ are connected to at least one tail, and thus has nim-value $1\oplus3=2$. 

Finally we make sure that $G$ cannot be moved to $3$. All moves in $C_G$ go to $0$, $1$, or $2$. Deleting a leaf of $G$ gives $0$ or $4$ by the inductive hypothesis. Deleting any other edge or vertex outside $C_G$ gives the combination of a smaller hairball (possible nim-values $0,3,4$) and a path graph ($1$ and $2$), which is also not $3$. This completes the induction. 
\end{proof}

\begin{corollary} 
Suppose graph $G$ has $V$ vertices and $V+1$ edges, where $V-1$ of the vertices have degree $2$ and the remaining vertex $A$ has degree $4$, as in Figure~\ref{cor1}. Then $g(G)=1$.
\end{corollary}

\begin{figure}[htbp]
\begin{center}
\includegraphics{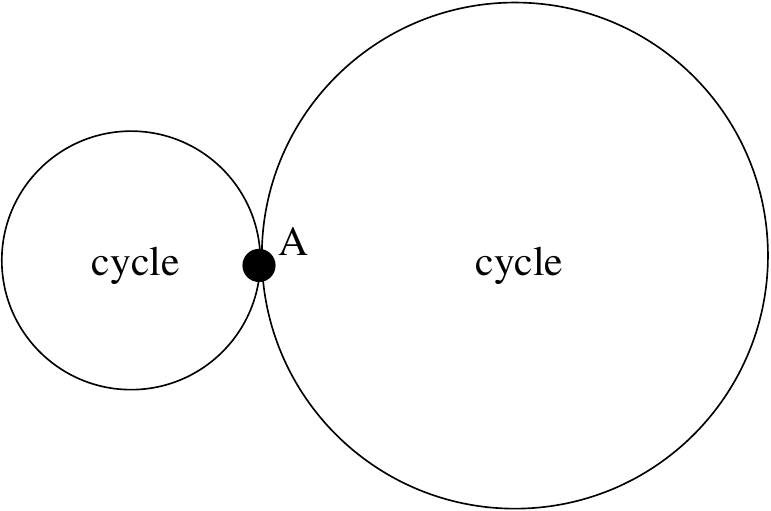}
\end{center}
\caption{Two cycles joined by a vertex.}
\label{cor1}
\end{figure}

\begin{proof}
Deleting $A$ gives a forest with nim-value either $0$ or $3$. Deleting any other vertex or edge results in a hairball, which never has nim-value $1$.  Moving to a position with nim-value $0$ is always possible, either by deleting $A$ if $V$ is odd or by deleting an edge which leaves an even-cycle pseudotree with an even number of vertices if $V$ is even. This proves the corollary.
\end{proof}

\begin{corollary}
Suppose graph $G$ can be constructed by taking two nonadjacent vertices in a cycle graph $C$ and adding another path between them, as in Figure~\ref{cor2}. Let $A$ and $B$ be the two vertices of degree $3$. Then $g(G)=1$ if $v(G)$ is odd and $g(G)=2$ if $v(G)$ is even.
\end{corollary}

\begin{figure}[htbp]
\begin{center}
\includegraphics{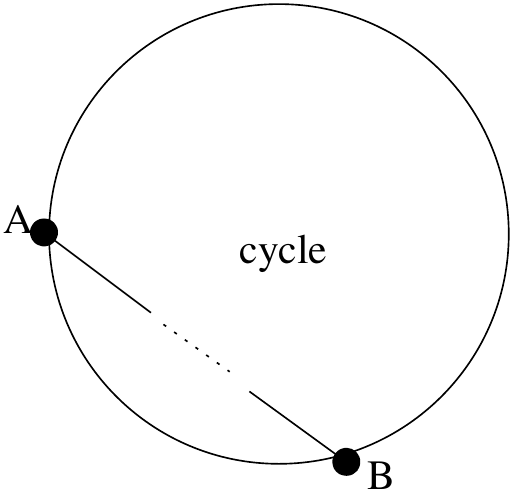}
\end{center}
\caption{A cycle with two nonadjacent vertices joined by any extra path.}
\label{cor2}
\end{figure}

\begin{proof}
Any move other than $A$ or $B$ results in a hairball, which never has nim-value $1$ or $2$. Deleting $A$ or $B$ results in a tree with nim-value $2$ if $v(G)$ is odd and $1$ if $v(G)$ is even. $G$ contains at least one even cycle, since if $C$ is not even the other two cycles are of opposite parity; thus it is always possible to move to an even-cycle pseudotree with an even number of vertices and nim-value $0$. This proves the corollary.
\end{proof}


\small
{\sc Tirasan Khandhawit, Department of Mathematics, M.I.T., Cambridge, MA 02139}

E-mail address: \url{tirasan@math.mit.edu}

{\sc Lynnelle Ye, 531 Lasuen Mall, P.O. Box 16820, Stanford, CA 94309}

E-mail address: \url{lynnelle@stanford.edu}

\end{document}